\documentclass[10t,a4paper]{amsart}

\usepackage[english]{babel}
\usepackage[T1]{fontenc}
\usepackage[utf8]{inputenc}

\usepackage{amscd,amssymb,amsopn,amsmath,amsthm,graphics,amsfonts,enumerate,verbatim,calc}
\usepackage{amssymb, tikz}
\usepackage{mathtools}

\usepackage{hyperref}
\usepackage[capitalise]{cleveref}

\usepackage{setspace}
\newtheorem{theorem}{Theorem}[section]
\newtheorem{lemma}[theorem]{Lemma}

\newtheorem{proposition}[theorem]{Proposition}
\newtheorem{corollary}[theorem]{Corollary}

\theoremstyle{definition}
\newtheorem{definition}[theorem]{Definition}
\theoremstyle{remark}
\newtheorem{remark}[theorem]{Remark}
\newtheorem{fact}[theorem]{Fact}
\newtheorem{example}[theorem]{Example}
\newtheorem{observation}[theorem]{Observation}

\newtheorem{conjecture}[theorem]{Conjecture}

\newcommand{\HH}{\operatorname{H}}

\newcommand{\Ass}{\operatorname{Ass}}

\newcommand{\CMdef}{\operatorname{CMdef}}
\newcommand{\pdef}{\operatorname{pdef}}
\newcommand{\grade}{\operatorname{grade}}
\newcommand{\Spec}{\operatorname{Spec}}

\newcommand{\rad}{\operatorname{rad}}

\newcommand{\Ht}{\operatorname{ht}}
\newcommand{\pd}{\operatorname{p.dim}}

\newcommand{\G}{\operatorname{G}}

\newcommand{\Gdim}{\operatorname{Gdim}}

\newcommand{\id}{\operatorname{id}}
\newcommand{\Ext}{\operatorname{Ext}}

\newcommand{\Supp}{\operatorname{Supp}}
\newcommand{\Tor}{\operatorname{Tor}}
\newcommand{\Hom}{\operatorname{Hom}}

\newcommand{\Ann}{\operatorname{Ann}}

\newcommand{\depth}{\operatorname{depth}}

\newcommand{\lo}{\longrightarrow}
\newcommand{\fm}{\frak{m}}
\newcommand{\fp}{\frak{p}}
\newcommand{\fq}{\frak{q}}

\begin{document}

\author[]{Mohsen Asgharzadeh}

\address{}
\email{mohsenasgharzadeh@gmail.com}

\title[ ]
{ A note on Cohen-Macaulay descent }

\begin{abstract}
We investigate the behavior of  Cohen-Macaulay defect undertaking tensor product with a perfect module.  Consequently, we study the perfect defect of a module.
As an application, we connect to associated prime ideals of tensor products.
\end{abstract}

\maketitle

\section{Introduction}

  We start with:
  
\begin{conjecture}(See \cite[Conjecture 2.11]{y}) 
	\label{1}Let $R$ be a Cohen-Macaulay  local ring, $M$ be a perfect $R$-module of positive dimension and let  $N$ be  a 
 module with $\dim (N )= \dim( R)$. If $M \otimes_R N$ is  a Cohen-Macaulay $R$-module, then $N$ is a maximal
	Cohen-Macaulay $R$-module.
\end{conjecture}

 Yoshida proved Conjecture \ref{1} under the additional assumption that $N$ is surjective Buchsbaum. 
Also, without loss of the generality, we may assume that $d:=\dim R\geq 2$. 

 In Example \ref{e} we give a negative answer to the conjecture over any Cohen-Macaulay local ring of dimension $d>2$.
  The module $M$ in Example \ref{e} is of projective dimension at least two.  In Section 3
  we show:
 \begin{proposition} 
 	Let $R$ be a Cohen-Macaulay local  ring with isolated singularity of dimension $d$, $M$ be perfect of projective dimension at most one, and let $N$ be  torsion-free. If $M\otimes_R N$ is 	Cohen-Macaulay, then  $N$ is maximal Cohen-Macaulay.
 \end{proposition}
 
  Also, we present some partial positive answers in the case $\pd(M)=2$ and give the corresponding obstruction.
  As an application, and in Section 4, 
 we compute Cohen-Macaulay defect of tensor product,  and compare it with the perfect defect.
 It may be interesting to find the situation for which these defects are coincide to each other. We determine such a situation,  see Corollary \ref{cmc}.  Despite of  simplicity  of $\Ass(\Hom(-,\sim))$ and $\Supp(-\otimes_R\sim)$,  it seems $\Ass(-\otimes_R\sim)$ rarely  computed in  literature, even in some special forms.
 For an achievement, see \cite[Theorem 23.2]{mat} where one of the modules is relatively flat. As a typical example, see  e.g. \cite[Corollary 1.6]{a}.
 Here,
 we show:

 \begin{corollary} \label{1.3}
 Let $R$ be local,  $M$  of finite projective dimension  and $N$ be maximal  Cohen-Macaulay. 
 The following holds: 	
\begin{enumerate}
 \item[i)] $\Ass(M\otimes_R N)\subseteq \Ass(M)\cap\Supp(N)$.
 \item[ii)]  If $R$ is Cohen-Macaulay, then $\Ass(M\otimes_R N)= \Ass(M)\cap\Supp(N)$. 
\item[iii)] If $M$ becomes an ideal  of positive height, then   $\Ass(M\otimes_R N)=\Ass(N)$.
\item[iv)] If $M$  is Cohen-Macaulay then  $\Ass(M\otimes_R N)= \Ass(M)\cap\Supp(N)$. In particular, $\Ass(M\otimes_R N)=\Ass(M)$ provided  
$N$ is fully supported (e.g. $R$ is a  domain).	
 \item[v)] Assume in addition to ii) that $R$ is of prime characteristic and $F$-finite, then $\Ass(F^n(M))=\Ass(M)$.
\end{enumerate}	
 \end{corollary}

Part  v) is well-known and presented here because of the motivation,  it is true without Cohen-Macaulay and $F$-finite assumptions, see e.g. \cite[Corollary 1.6]{HS} by Huneke-Sharp.
Under the extra assumption $M$ is primary, i.e.  $\Ass(M)$ is singleton, Corollary \ref{1.3}(i) implies that
$\Ass(M\otimes_R N)=\Ass(M)$. 
Also, we present a  connection to the support of  maximal  Cohen-Macaulay modules, see Corollary \ref{supp} and \ref{crew}.	Let  $M$ be of finite projective dimension over a local ring, and let $\Sigma$ be the class
of all maximal Cohen-Macaulay modules. It follows by the above corollary that $\bigcup_{L\in\Sigma}\Ass(M\otimes_R L)$ is finite.

Suppose the module $M$ from Corollary \ref{1.3}(i) is primary and cyclic. The existence  of  such a
module may put  some restrictions on the ring. In fact,
Peskine and Szpiro    \cite[Corollary II.3.3]{PS2}   proved that the existence  of  a prime ideal
of finite projective dimension implies the ring is an integral domain.
The corresponding infinitesimal version may not be true.
For example, let $R:=k[X]/(X^2)$ and $\fp:=(x)$. Then $\fp^2=\fp^{(2)}=0$, and consequently  $\fp^{(2)}$
is free. But, $R$ is not an integral domain. Despite of this we have:

\begin{observation} 
	Let $R$ be a  local  ring,  $\fp \in\Spec(R)$ such that $\pd(R/\fp^{(n)}) < \infty$ for some  $n\in\mathbb{N}$. Adopt one of the following situations:
	 	\begin{enumerate}
		\item[i)] $\fp^{(n)}\neq 0$, or
		\item[ii)]   $R$ is reduced.	
	\end{enumerate}	
	 Then  $R$ is an integral domain.
\end{observation}

It  may be worth to mention that the dual version of this observation   (working with injective dimension)  imposes
more restrictions on the ring. For the precise statement, see Corollary \ref{id}.

 
\section{Negative side of  Conjecture \ref{1}}

The notation $(R,\fm,k)$ stands for a   commutative, noetherian and local ring, with the maximal ideal $\fm$ and  the residue field  $k$. Also, modules are finitely generated, otherwise specialized.
  The notation $\pd_R(-)$ stands for projective dimension of an $R$-module $(-)$. By $\grade_R (M)$ we mean $\grade_R (\Ann M, R),$ i.e.,  $$\grade_R (M)=\inf\{i:\Ext^i_R(M,R)\neq 0\}.$$
Recall that a non-zero  module
$M$ is called perfect if $\pd_R (M) = \grade_R (M)$. So, $\pd_R (M) < \infty$.
For simplicity of the reader, we bring the following well-known and easy facts:

\begin{fact}
	Let $f: R \to S$ be a homomorphism of  local rings,
	and $M$ an $S$-module which is finite as an $R$-module.	Then $\depth_R(M)=\depth_S(M)$.
\end{fact}

\begin{fact}\label{22}
Let  $R$ be a Cohen-Macaulay local ring, and $M$ be 
of finite projective dimension.
Then $M$ is perfect if and only if  $M$  is Cohen-Macaulay.
\end{fact}

A module $M$ is called \it{quasi-perfect}  if
$\inf\{i:\Ext^i_R(M,R)\neq 0\}=\sup\{i:\Ext^i_R(M,R)\neq 0\}$.
The monograph \cite{AB} is our reference for the concept of  $\G$-dimension.

\begin{observation} 
	Let $R$ be a	 local ring and $M$ be Cohen-Macaulay module of finite $\G$-dimension. Then $\dim (M)+\grade(M)=\depth (R)$
	and $M$ is quasi-perfect.
\end{observation}

\begin{proof}
	 Recall from \cite[Lemma I.4.8]{PS2} that  $\dim M+\grade(M)\geq\depth(R)$.
	Also, in the case  $\G$-dimension is finite we have	$\Gdim_R M=\sup\{i:\Ext^i_R(M,R)\neq 0\}$, see  \cite{AB}. In particular,
	$\grade(M)\leq\Gdim(M)$. We apply these along with Auslander-Bridger formula  to observe that
	\[\begin{array}{ll}
	\dim M&\geq\depth(R)-\grade(M)\\
	&\geq\depth(R)-\Gdim(M)\\
	&=\depth_R (M)\\
	&=\dim (M).
	\end{array}\]
	So, $\dim (M)+\grade(M)=\depth(R)$ and 	$\grade(M)=\Gdim(M)$, i.e., $M$ is quasi-perfect.
\end{proof}

By
$\HH^i_{\fm}(-)$ we mean the $i$-th local
cohomology module of $(-)$ with respect to $\fm$.

\begin{example}\label{e}
	Let $(R,\fm)$ be a Cohen-Macaulay local ring of dimension $d>1$ and let $x_1,\ldots, x_d$
	 be a system of parameters and let $N:=I:=(x_1,\ldots,x_{d-1})$. The following holds: 	
\begin{enumerate}
\item[i)] $M:=R/I$ is perfect,
\item[ii)]   $\dim (N )= \dim (R)$ and  $M \otimes_R N$ is Cohen-Macaulay as an $R$-module,
\item[iii)] if $d>2$, then  $N$ is not maximal Cohen-Macaulay as an $R$-module,
\item[iv)] if $d=2$, then 	$N$ is   maximal 	Cohen-Macaulay as an $R$-module.
	\end{enumerate}
\end{example}

\begin{proof} i):	Clearly, $M$ resolves by the Koszul complex with respect to $x_1,\ldots,x_{d-1}$.
		In view of Fact \ref{22}  $M$ is perfect, because it is Cohen-Macaulay. 
		
	ii):
Let $\fp\subset I$ be a minimal prime ideal. It follows that $N$ is 
	of dimension equal to $\dim R$, because $\fp\in\Ass(N)$ and $R$ is equidimensional.
	 Since $I$ is generated by a regular sequence, $\frac{I}{I^2}$ is free as 
	 an $R/I$-module. It is easy to see $R/I$ is Cohen-Macaulay.
	From this, $\frac{I}{I^2}$ is Cohen-Macaulay as an $R/I$-module. Note
	that dimension of $\frac{I}{I^2}$ does not change as an $R$-module. We apply Fact 2.1 to deduce that
	$\frac{I}{I^2}$ is Cohen-Macaulay as an $R$-module. Note that $M \otimes_R N=I/I^2$.

iii): Suppose $d>2$.	Here, we claim that $N$ is not Cohen-Macaulay. Indeed, look at $0\to I\to R\to R/I\to 0$.
		In view of the long exact sequence of local cohomology
		modules we have $$\HH^1_{\fm}(R)\lo \HH^1_{\fm}(R/I)\lo\HH^2_{\fm}(I)\quad(\ast)$$
		Due to Grothendieck's non-vanishing theorem, we know  $\HH^1_{\fm}(R/I)\neq 0$.
		By local cohomology characterization of depth, $\HH^1_{\fm}(R)=0$. We put this
		into $(\ast)$ and we deduce that $0\neq \HH^1_{\fm}(R/I)\subset\HH^2_{\fm}(I).$
	 By local cohomology characterization of depth, $\depth(I)\leq 2<d$. Since $d=\dim( N)$
	we observe that  $N$  is not Cohen-Macaulay, viewed as an $R$-module.
	
iv):	Suppose $d=2$. Note that $M \otimes_R N=N/x_1N$. Since $x_1$ is regular over $A$, it is regular over its
 submodules. In particular, $x_1$ is $N$-regular.  Recall from  item ii) that  $N/x_1N$  is Cohen-Macaulay as an $R$-module. From these, $N$ is Cohen-Macaulay as an $R$-module. 
 Since $\dim (N)=d$, and by definition, $N$ is maximal
	Cohen-Macaulay as an $R$-module.
\end{proof}

By $\Tor^R_+(M,N)$ we mean $\oplus_{i>0}\Tor^R_i(M,N)$. The following inspired us to present \S 3.
\begin{remark}\label{tor}
	Adopt the notation of Example \ref{e}. The following holds:
	\begin{enumerate}
		\item[i)]  if $d>2$, then  $\Tor^R_+(M,N)\neq0$,
		\item[ii)] if $d=2$, then 	$\Tor^R_+(M,N)=0$.
	\end{enumerate}
Indeed, assume $d>2$ and  apply $-\otimes_R I$ to $0\to I\to R\to R/I\to 0$, we obtain $$0\to \Tor^R_1(M,N)\to I\otimes I \stackrel{f}\lo I\to I/I^2\to 0,$$where $f(x\otimes y)=xy$. In particular, $\Tor^R_1(M,N)	=\ker(f)$ and $0\neq x_1\otimes x_2-x_2\otimes x_1\in\ker(f)$. From this, $\Tor^R_+(M,N)\neq0$.
In the case $d=2$,   we see $N$ is free, and so claim is trivial.
\end{remark}


\section{Positive side of Conjecture \ref{1}}

Conjecture \ref{1} is true over $1$-dimensional rings  by a trivial reason:

\begin{observation}\label{one}
	Let $R$ be a $1$-dimensional local  ring, $M$
	be  perfect and of positive dimension, and let $N$ be  $1$-dimensional. If $M\otimes_R N$ is 
	Cohen-Macaulay, then  $N$ is maximal Cohen-Macaulay.
\end{observation}

\begin{proof}
	We know that $\depth(M)=\dim(M)>0$. By Auslander-Buchsbaum
	formula, $\pd(M)=0$. Recall that  $M\otimes_R N=\oplus N$ is 
	Cohen-Macaulay. From this $N$ is Cohen-Macaulay, and so the claim follows.
\end{proof}

The module $M$ in Example \ref{e}(iii) is of projective dimension at least two.
What can say if $\pd_R (M)=1$? First, we recall the following two results:

\begin{fact}\label{aus}(See \cite[Theorem 1.2]{au})
 Let $R$ be any local ring and $\pd_R (M)<\infty$. Let $q$ be the largest number such that  $\Tor_q^R(M, N)\neq0$. If $\depth_R (\Tor_q^R(M, N))\leq1$ or $q=0$,
	then $$\depth_R(N)=\depth_R (\Tor_q^R(M, N))+\pd_R(M)-q.$$
\end{fact}

\begin{fact}\label{YD}
	Let $R$ be local,  $M$  of finite projective dimension  and $N$ be such that $\dim (N)=\dim (R)$. The following holds: 	
	\begin{enumerate}
		\item[i)] $\depth(M)\leq\dim(M\otimes_R N)\leq \dim (M)$,
		\item[ii)] if $M$ is perfect, then
	$\dim(M\otimes_RN)=\dim (M)$.
	\end{enumerate}
\end{fact}

\begin{proof}
i): The second inequality always holds, because  $\Supp(M\otimes _R N)=\Supp(M)\cap\Supp(N)$. 
	By intersection theorem \cite{int}, $$\dim(M\otimes_R N)\geq\dim (N)-\pd(M)=\dim(R)-\pd(M)\geq\depth(R)-\pd(M).$$
	By Auslander-Buchsbaum, $\depth(M)\leq\dim(M\otimes_R N)$.
	
	ii): This is in \cite[Lemma 2.1]{y}.
\end{proof}

\begin{example}\label{ndformula}
Concerning Fact \ref{YD}, the first items shows that the assumption $\pd(M)<\infty$ is really needed   even if we assume $M$ is quasi-perfect. The second item shows that the  dimension restriction on $N$ is important:
\begin{enumerate}
		\item[i)] Let $R:=\frac{k[[x,y]] }{(xy)}$, $M:=R/(x)$ and $N:=R/(y)$. Then $\dim(M)=\dim(N)=\dim(R)=1$, 
	$M$ is maximal Cohen-Macaulay and so   totally reflexive. But, 
$\dim(M\otimes_RN)=0<\dim (M)$.
\item[ii)] Let $(R,\fm)$ be a $d$-dimensional regular local ring with $d>1$.
There are $d$ elements $\{x_1,\ldots,x_d\}$ that generates $\fm$.  Let $s\in[1,d-1]$ be an integer, define $\fp:=(x_1,\ldots,x_s)$ and 
$\fq:=(x_{s+1},\ldots,x_d)$. Clearly,
$M:=R/\fp$ is perfect and let $N:=R/ \fq$. Then $\dim(M\otimes_RN)=0<\depth(M)=\dim (M)=s$.  
\end{enumerate}
\end{example}

In the case of integral domains, one can   say a little more:

\begin{remark}
	Let $R$ be a local domain, $M,N$ be finitely generated  such that $\dim (N)=\dim (R)$. Then $\Supp(M\otimes _R N)=\Supp(M)$. In particular, $\dim(M\otimes_RN)=\dim M$.
\end{remark}

\begin{proof}
Since  both modules are finitely generated, $\Supp(M\otimes _R N)=\Supp(M)\cap\Supp(N)$. Since $R$ is domain, $\Supp(N)=\Spec(R)$.
	From this,  $\Supp(M\otimes _R N)=\Supp(M)$.
\end{proof}

 Here, we present a partial positive answer to Conjecture \ref{1}:

\begin{proposition}\label{pone}
Let $R$ be a Cohen-Macaulay local  ring with isolated singularity of dimension $d$, $M$
 be perfect of projective dimension at most one, and let $N$ be  torsion-free. If $M\otimes_R N$ is 
 	Cohen-Macaulay, then  $N$ is maximal Cohen-Macaulay.
\end{proposition}

\begin{proof}
	Since $N$ is torsion-free, it is of dimension $d$. We may assume that $d>0$. The proof is by induction on $d$. If $d=1$
	we get the claim because $\depth_R(N)>0$. By applying the inductive step, it turns out that 
	$\ell(\Tor^R_1(M,N)) <\infty$, because $R$ is regular
	over the punctured spectrum and that maximal  Cohen-Macaulay modules over regular rings are free.
	Let $0\to R^n\to R^m\to M\to 0$ be a free resolution of $M$. Apply   $-\otimes_RN$ to it we have
	 the following exact sequence $$0\lo \Tor^R_1(M,N)\lo N\otimes _R R^n\lo N\otimes_R  R^m \lo M\otimes_RN\lo 0.$$
	  Since $N$ is of positive depth we deduce that  	$\Tor^R_1(M,N)=0$.
	   Thus, $\Tor^R_+(M,N)=0$. This and finiteness of $\pd_R (M)$ allow us to
	apply the depth formula of Auslander:
	$$\depth_R (N)=\depth_R (M\otimes_R  N)+\pd_R (M)\quad(\ast)$$
 Also, in the light of Auslander-Buchsbaum formula, $$\depth_R (M)+\pd_R (M)=\depth(R)\quad(+)$$
Then, we have
	
	 \[\begin{array}{ll}
\depth_R (N)&\stackrel{(\ast)}=\depth_R (M\otimes_R  N)+\pd_R (M) \\
&=\dim (M\otimes_R  N)+\pd_R (M)\\
	&\stackrel{\ref{YD}}=\dim (M)+\pd_R (M) \\
	&\stackrel{\ref{22}}=\depth_R (M)+\pd_R (M) \\
	&\stackrel{(+)}=\depth(R)\\
	&=\dim (R).
	\end{array}\]
	 By definition, $N$ is maximal Cohen-Macaulay.
\end{proof}

In the same vein we have
\begin{proposition}
	Let $R$ be a Cohen-Macaulay local ring of dimension $d$, $M$ be perfect of projective dimension at most one, 
	and let $N$ be   torsion-free and locally free over the punctured spectrum. If $M\otimes_R N$ is 
		Cohen-Macaulay, then  $N$ is maximal Cohen-Macaulay.
\end{proposition}

Note that Conjecture \ref{1} is not true when $\pd(M)=2$. Despite of this,
we have:
\begin{remark}
	Let $R$ be a Cohen-Macaulay local  with isolated singularity of dimension $d$. Let $M$ be
	a cyclic perfect module of projective dimension at most two, and let $N$ be   torsion-free.
	 If $M\otimes_R N$  is	Cohen-Macaulay and   $(\Ann_RM)\otimes_R N $ is torsion-free, then  $N$ is maximal Cohen-Macaulay.
\end{remark}

\begin{proof}
	We may assume that $d>0$. The proof is by induction on $d$. If $d=1$
	we get the claim because $\depth_R(N)>0$. By applying the inductive step, it turns out that 
	$\ell(\Tor^R_+(M,N)) <\infty$. Let $I:=\Ann_R(M)$. In the light of Hilbert-Burch, the free resolution of $I$ is of the following from
	 $0\to R^n\to R^{n+1}\to I\to 0.$    Apply   $-\otimes_RN$ to it we have the following
	  exact sequence $$0\lo \Tor^R_1(I,N)\lo N\otimes _R R^n\lo N\otimes_R  R^{n+1} \lo M\otimes_RN\lo 0.$$ 
	  Since $N$ is of positive depth we deduce that  	$\Tor^R_1(I,N)=0$. From this, $\Tor^R_2(M,N)=0$. 
	   We apply   $-\otimes_RN$ to  $0\to I\to R\to R/I\to 0$  we have the following exact sequence
	    $$0\lo \Tor^R_1(R/I,N)\lo I\otimes _RN \lo N \lo M\otimes_RN\lo 0.$$
	By our assumption, $\depth( I\otimes_R N )>0$. From this, $\Tor^R_1(R/I,N)=0$.
	 Thus, $\Tor^R_+(M,N)=0$.  Similar to Proposition \ref{pone} we have
	
	\[\begin{array}{ll}
	\depth_R (N)&=\depth_R (M\otimes_R  N)+\pd_R (M) \\
	&=\depth_R (M)+\pd_R (M) \\
	&=\depth(R).
	\end{array}\]
Since $\depth(R)=\dim R$, $N$ is maximal Cohen-Macaulay.
\end{proof}

Keep Remark \ref{tor} in mind and recall that $\pd(M)=\dim R-1$. This remark motivate us
to restate Conjecture \ref{1} in the following obstruction format:
\begin{remark}
Let $R$ be a Cohen-Macaulay  local ring, $M$ be a perfect $R$-module of positive dimension and let  $N$ be  a 
module with $\dim (N) = \dim (R)=d$. If $M \otimes_R N$ is  a Cohen-Macaulay $R$-module. Suppose
in addition $\pd(M)=d-1$. Then Conjecture \ref{1} is true if and only if $\Tor^R_+(M,N)=0$.	
	\end{remark}

\begin{proof}
	Suppose $\Tor^R_+(M,N)\neq0$.	Then
	$q:=\sup\{ i:\Tor_i^R(M, N)\neq0\}\geq1$. By Auslander-Buchsbaum formula, $\depth(M)=1$.
	Recall that  $\depth_R (\Tor_q^R(M, N))\leq\dim (\Tor_q^R(M, N))\leq\dim(M)=1$. In the light of Fact \ref{aus} we see
	 $$\depth_R(N)=\depth_R (\Tor_q^R(M, N))+\pd_R(M)-q\leq 1+(d-1)-1\leq d-1,$$ i.e.,
	 $N$ is not maximal Cohen-Macaulay. The reverse part is in the proof of Proposition \ref{pone}.
\end{proof}

\section{Perfect and Cohen-Macaulay defects}

The following  drops finiteness of $N$ from \cite[Lemma 2.2]{y}. 

\begin{lemma} \label{big}
	Suppose that $\pd_R(M) < \infty$ and $N$ is   big Cohen-Macaulay.  Then $\Tor^R_+(M,N) = 0$. \end{lemma}

\begin{proof}
	Without loss of the generality we may assume that $R$ is complete. This enable us to assume 
	that $M$ is balanced    big Cohen-Macaulay. Now, the proof is similar to \cite[Lemma 2.2]{y}, only note that instead of the Eisenbud-Buchsbaum criterion
	we use a criterion of Northcott. 
\end{proof}
The third item of the following result reproves \cite[Proposition 2.4]{y},  where Yoshida used the intersection theorem and  a formula   involved  in Betti and Bass numbers.

\begin{proposition} \label{re}
	Let $R$ be local,  $M$  of finite projective dimension  and $N$ be maximal  Cohen-Macaulay. 
The following holds: 	
\begin{enumerate}
	\item[i)] $\depth(M)\leq\depth(M\otimes_R N)\leq \dim (M)$,
	\item[ii)] if $R$ is  Cohen-Macaulay, then $\depth(M)=\depth(M\otimes_R N)$,	
	 \item[iii)] if $M$ is perfect, then
	 $M\otimes_R N$  is	Cohen-Macaulay.
\end{enumerate}			
\end{proposition}

\begin{proof}
By Lemma \ref{big}, $\Tor^R_+(M,N)=0$. This allow us to apply Auslander's depth formula.
In view of Fact \ref{aus}  we observe that	
	$$\depth_R (M\otimes_R  N)=\depth_R (N)-\pd_R (M)=\dim (R)-\pd(M)\quad(\ast)$$
	 i): By Auslander-Buchsbaum formula we have:
	 	$$\depth_R (M\otimes_R  N)\stackrel{(\ast)}=\dim( R)-\pd(M)\geq \depth(R)-\pd(M)=\depth(M).$$
	To see the second inequality, we note that 
	$$\depth(M\otimes_R N)\leq \dim(M\otimes_R N)\leq\dim (M)\quad(+).$$

ii):	Since $R$ is Cohen-Macaulay, we have $\dim( R)=\depth(R)$. Put this
in $(\ast)$ and apply Auslander-Buchsbaum formula to see $$\depth_R (M\otimes_R  N)=\depth(R)-\pd(M)=\depth(M).$$
	
	iii):
Recall  that $\dim(R)=\dim(M)+\grade(M)$, because $M$ is perfect. 
Then
	\[\begin{array}{ll}
\dim (M)&\stackrel{(+)}\geq\depth_R (M\otimes_R  N) \\
&\stackrel{(\ast)}=\dim(R)-\pd(M)\\
&=\dim(R)-\grade(M)\\
&=\dim(M).
\end{array}\]Therefore, $\dim (M) =\depth_R (M\otimes_R  N)$. Since $\dim (M\otimes_R N)\leq \dim (M)$,  it follows that $M\otimes_R  N$ is  Cohen-Macaulay.
\end{proof}

\begin{remark}
	Proposition \ref{re}(i) is not true even if we assume  $\Gdim (M)<\infty$, see
	Example \ref{ndformula}. 
\end{remark}

Here, we present a series of corollaries and applications (also, see the next section):

\begin{corollary} \label{cre}
	Let $(R,\fm)$ be Cohen-Macaulay,  $M$  of finite projective dimension  and $N$ be maximal  Cohen-Macaulay. 
	The following holds: 	
\begin{enumerate}
\item[i)] $\Ass(M\otimes_R N)=\Ass(M)\cap\Supp(N)$.
\item[ii)] If  $R$ is  a domain, then 	$\Ass(M\otimes_R N)=\Ass(M)$.
\end{enumerate}	
\end{corollary}

\begin{proof}
	i): Let $\fp\in\Ass(M\otimes_R N)$. First, we claim that $N_{\fp}$ is maximal  Cohen-Macaulay.
	Indeed, $\fp\in\Supp(M\otimes_R N)\subseteq\Supp(N)$. Since $R$ is Cohen-Macaulay, there is an $R$-sequence
	$x_1,\ldots,x_h$ of length $h:=\Ht(\fp)$. Extend this to a full system of parameters of $R$ and call it $\underline{x}$.
	Recall that $N$ is maximal Cohen-Macaulay. This implies that  $\underline{x}$
	is $N$-sequence.
	From this, $x_1/1,\ldots,x_h/1$ is $N_{\fp}$-sequence. Consequently, $N_{\fp}$ is maximal Cohen-Macaulay over
	$R_{\fp}$. In view of Proposition \ref{re}(ii) we see that $$\depth(M_{\fp})=\depth(M_{\fp}\otimes_{R_{\fp}} N_{\fp})\quad(\ast)$$
	Since $\fp\in\Ass(M\otimes_R N)$, it turns out that $\depth(M\otimes_R N)_{\fp}=\depth(M_{\fp}\otimes_{R_{\fp}} N_{\fp})=0$.
	We put this in $(\ast)$ and deduce that $\depth(M_{\fp})=0$. This is equivalent to saying that $\fp\in\Ass(M)$.
	The reverse inclusion holds by the same reasoning.
	
	ii):   In the case  $R$ is  a domain, we know $\Supp(N)=\Spec(R)$. So,
the desired claim is an immediate application of i).
\end{proof}

Let us to recover:

 \begin{fact} (Huneke-Sharp)
	Let $(R,\fm)$ be Cohen-Macaulay  of prime characteristic $p$ and $F$-finite,  $M$  of finite projective dimension. 
	Then $\Ass(F^n(M))=\Ass(M)$.
\end{fact}

\begin{proof}
	Each iteration of Frobenius  defines a new $R$-module structure on $R$, and this $R$-module is denoted by $ { ^n}R$.
If $x_1,\ldots,x_d$ is a full system of parameters, we know $x_1^{p^n},\ldots,x_d^{p^n}$ is a regular sequence.
From this, ${ ^n}R$ is maximal  Cohen-Macaulay and it is of full support. In view of Corollary \ref{cre}(i) $$\Ass(F^n(M))=\Ass(M\otimes { ^n}R)=\Ass(M)\cap\Supp({ ^n}R)=\Ass(M)\cap\Spec(R)=\Ass(M),$$as claimed.
\end{proof}

\begin{corollary}\label{can}
	Let $R$ be a  Cohen-Macaulay local ring  equipped with a canonical module $ \omega_R$ and let $M$ be   of finite projective dimension. 
	The following holds: \begin{enumerate}
		\item[i)]   $\Ass(M\otimes_R \omega_R)=\Ass(M)$.
		\item[ii)]  $M\otimes_R \omega_R$ has an $\omega_R$-resolution of length $\pd(M)$. 
	\end{enumerate}
 
\end{corollary}

 \begin{proof}
 i):	Recall that $\Supp(\omega_R)=\Spec(R)$. This  yields i), see Corollary \ref{cre}.
 
  ii): Let $p:=\pd(M)$. By definition, there is an exact sequence $0\to R^{\beta_p}\to\ldots\to R^{\beta_0} \to M\to 0$. Tensor it  with $-\otimes_R\omega_R$  we have $0\to \omega_R^{\beta_p}\to\ldots\to \omega_R^{\beta_0} \to M\otimes_R\omega_R\to 0$  which is exact by Lemma \ref{big}. By definition, $M\otimes_R \omega_R$ has an $\omega_R$-resolution of length $p$.
 \end{proof}

Concerning Corollary \ref{cre} a natural question arises: does maximal Cohen-Macaulay modules localize? Also, this was asked in MathOverflow, before than us.
Here, we affirmatively answer it:

\begin{corollary} \label{crew}
	Let $(R,\fm)$ be local,  $M$  of finite projective dimension  and $N$ be maximal  Cohen-Macaulay. 
	The following holds: 	
	\begin{enumerate}
		\item[i)] $N$ is locally maximal  Cohen-Macaulay over its support.
		\item[ii)]   	$\Ass(M\otimes_R N)\subseteq\Ass(M)\cap\Supp(N)$.
		\item[iii)] If $M$  is Cohen-Macaulay then  $\Ass(M\otimes_R N)= \Ass(M)\cap\Supp(N)$. In particular, $\Ass(M\otimes_R N)=\Ass(M)$ provided  
		$N$ is fully supported.	
	\end{enumerate}	
\end{corollary}

\begin{proof}
	i): Let $\fp\in\Supp( N)$. First,   $N_{\fp}$ is    Cohen-Macaulay and nonzero.
It turns out that $$\dim(R)=\dim (N)=\dim(N_{\fp})+\dim(\frac{N}{\fp N})\quad(1)$$ Also, $\Ann(\frac{N}{\fp N})\supseteq \fp$, and so
\begin{enumerate}
\item[(2)]  $\dim(N_{\fp }) \leq \dim(R_{\fp})=\Ht(\fp)$,             
\item[(3)]   $\dim(\frac{N}{\fp N})=\dim(\frac{R}{\Ann(N / \fp N)})\leq \dim(\frac{R}{\fp })$,
\item[(4)]   $\dim(\frac{R}{\fp })+\Ht(\fp)\leq \dim(R)$.
	\end{enumerate}	
Then we have
\[\begin{array}{ll}
\dim(R)&\stackrel{(1)}=\dim(N_{\fp})+\dim(\frac{N}{\fp N})\\
&\stackrel{(2)}\leq \Ht(\fp) +\dim(\frac{N}{\fp N})\\
& \stackrel{(3)}\leq \Ht(\fp)+\dim(\frac{R}{\fp }) \\
&\stackrel{(4)}\leq \dim(R).
\end{array}\]
In particular, $\dim(N_{\fp})=\Ht(\fp)$. Since $N_{\fp}$ is Cohen-Macaulay, it follows that
$N_{\fp}$ is  maximal  Cohen-Macaulay. 

ii): Apply the first item along with the proof of Corollary \ref{cre}
and use Proposition \ref{re}(i) to see 
$$\Ass(M\otimes_R N)\subseteq\Ass(M)\cap\Supp(N),$$
as claimed.

iii): Due to the second item, $\Ass(M\otimes_R N)\subseteq\Ass(M)\cap\Supp(N)$. In view of Proposition  \ref{re}(i)
 $\depth(M)\leq\depth(M\otimes_R N)\leq \dim (M)$. Since $M$ is Cohen-Macaulay,  $\depth(M)=\depth(M\otimes_R N)$.
Now, the argument similar to Corollary \ref{cre} shows that  $\Ass(M\otimes_R N)= \Ass(M)\cap\Supp(N)$.
\end{proof}

\begin{corollary}  
	Let $(R,\fm)$ be local,  $M$  of finite projective dimension. Let $\Sigma$ be the class
		of all maximal Cohen-Macaulay modules. Then $\bigcup_{L\in\Sigma}\Ass(M\otimes_R L)$ is finite.
\end{corollary}

\begin{proof}
In the light of Corollary \ref{crew}(ii) we observe that
$\bigcup_{L\in\Sigma}\Ass(M\otimes_R L)\subseteq \Ass(M)$. Since
	$M$ is finitely generated, the set of its associated prime ideals is finite. So,  $\bigcup_{L\in\Sigma}\Ass(M\otimes_R L)$ is finite.
\end{proof}

We say a module $M$ is  primary if  its zero submodule is primary in $M$, i.e., $\Ass(M)=\{\fp\}$.

\begin{corollary} \label{pri}
	Let $(R,\fm)$ be local,  $M$ be primary  of finite projective dimension  and $N$ be maximal  Cohen-Macaulay. 
	Then
	$M\otimes_R N$ is primary.	In fact, $\Ass(M\otimes_R N)=\Ass(M)$.  
\end{corollary}

\begin{proof}
	Let 	$\fp\in\Spec(R)$  be such that $\Ass(M)=\{\fp\}$. Recall that $\Ass(L)\neq \emptyset$ if and only if $L\neq 0$. In view of Corollary \ref{crew}(ii) we see $\emptyset\neq\Ass(M\otimes_R N)\subseteq \Ass(M)\cap\Supp(N)=\{\fp\}\cap\Supp(N)\subset \{\fp\}.$ So, $\Ass(M\otimes_R N)= \{\fp\}.$ By definition, 	$M\otimes_R N$ is primary.
\end{proof}

The following completes the proof of Corollary \ref{1.3} from the introduction.

\begin{corollary} \label{cre1}
	Let $R$ be a local ring,  $I\lhd R$  be of finite projective dimension  and $N$ be maximal  Cohen-Macaulay. 
	If $\Ht(I)>0$, then   $\Ass(I\otimes_R N)=\Ass(N)$.	
\end{corollary}

\begin{proof}
	 We look at $0\to I\to R\to R/I\to 0$ and recall from    Lemma \ref{big}  that $\Tor^R_1(R/I,N)=0$.
	We 	tensor the short exact sequence with $-\otimes_RN$ and obtain the following exact sequence $$0=\Tor^R_1(R/I,N)\lo I\otimes_R N\lo N\lo R/I\otimes_R N\lo 0.$$
	This yields $$\Ass(I\otimes_R N)\subseteq\Ass(N)\subseteq\Ass ( N/IN) \cup \Ass(I\otimes_R N)\quad(+)$$
	Recall from Corollary \ref{crew}(ii) that
	$\Ass(R/I\otimes_R N)\subseteq \Ass(R/I)\cap\Supp(N)$.
	We plug this $(+)$ and observe that
	\[\begin{array}{ll}
	\Ass(I\otimes_R N)\subseteq\Ass(N)& \subseteq( \Ass(R/I)\cap\Supp(N))  \cup \Ass(I\otimes_R N)\quad(\ast)
	\end{array}\]
	Recall that $N$ is equi-dimensional and of dimension equal to $\dim(R)$. Now, let $\fp\in\Ass(N)$. Since $\Ht(I)>0$,
	$\fp\notin \Ass(R/I)$. By $(\ast)$, $\fp\in\Ass(I\otimes_R N)$, i.e., $\Ass(N)\subseteq \Ass(I\otimes_R N)$.
	In sum, $$\Ass(I\otimes_R N)\subseteq\Ass(N)\subseteq\Ass(I\otimes_R N),$$and the desired claim  follows.
\end{proof}

\begin{example}
	Concerning Corollary \ref{cre1}, the first  item  shows that  the assumption $\pd(I)<\infty$ is needed even if we assume $\Gdim(I)<\infty$.
	The second item shows that the   maximal  Cohen-Macaulay  assumption  is needed. 
	\begin{enumerate}
		\item[i)]  Let $(R,\fm,k)$ be any one-dimensional local integral domain which is not regular. Let $I:=\fm$ and look at the  maximal  Cohen-Macaulay module
		$N:=\fm$. We claim  that $\Ass(I\otimes_R N)\neq\Ass(N)$.  Indeed, we look at $0\to I \to R\to k\to 0$
		and we drive the following exact sequence $0\to \Tor_1^R(k,N)\to I\otimes_R N\to N\to N\otimes_R k\to 0$.
		Recall that   and $\Tor_1^R(k,N)\simeq\Tor_2^R(k,R/I)\simeq k^{\beta_2(R/I)}\neq 0$. Since
		this nonzero module is annihilated by $\fm$, we deduce that
		$\HH^0_{\fm}(I\otimes_R N)\neq 0$, i.e., $\fm\in\Ass(I\otimes_R N)$. This yields the first claim.
		If we assume in addition that $R$ is Gorenstein, then $\Gdim(I)=0<\infty$.
		\item[ii)] Let $(R,\fm)$ be any regular local ring of dimension greater than one. Let $I:=\fm$ which is of finite
		projective dimension and look at
		$N:=\fm$. Similar to the first item, we have  $\Ass(I\otimes_R N)\neq\Ass(N)$. 
	\end{enumerate}
\end{example}

It may be nice to know  when  a prime ideal is in the support of a  maximal  Cohen-Macaulay module:

\begin{corollary} \label{supp} 
	Let $R$ be   local   and $N$ be maximal  Cohen-Macaulay. 
	Then  $$\{\fp\in\Spec(R):\exists0\neq I\lhd R\emph{ s.t. }  \pd(I)<\infty \emph{ and } \Ass(\frac{R}{I}) =\{\fp\}\}\subseteq\Supp(N)\setminus\Ass(R).$$ 
\end{corollary}

\begin{proof} 
	Let 	$\fp\in\Spec(R)$ be such that there is a  $\fp$-primary ideal $I	$ 
	of finite projective dimension.
	Set $M:=R/ I$ and recall that $\emptyset\neq\Ass(M\otimes_R N)\subseteq \{\fp\}\cap\Supp(N).$  So, $\fp\in\Supp(N)$.
	In order to prove $\fp\not\in\Ass(R)$ we remark that $\cup_{\fq\in\Ass(R)}\fq$
	consists of zero-divisors of $R$ and that any non-trivial ideal of finite projective dimension
	equipped with  a regular element.
\end{proof}

To see maximal  Cohen-Macaulay  modules are not necessarily full supported, we look at 
$R:=k[[x,y]]/(xy)$ and the   maximal  Cohen-Macaulay module  $N:=R/(x)$. Then $\Supp(N)=\{(x),\fm\}\subsetneqq\{(x),(y),\fm\}
=\Spec(R)$. In this example,$$\{\fp\in\Spec(R):\exists I\lhd R\emph{ s.t. }  \pd(I)<\infty \emph{ and } I \emph{ is }\fp-primary\}=\Supp(N)\setminus\Ass(R).$$ 

\begin{definition}
	By Cohen-Macaulay defect we mean $\dim(M)-\depth(M)$
and we denote it by $\CMdef(M)$.
\end{definition}

\begin{corollary}\label{ap}
	Let $R$ be Cohen-Macaulay, $M$  of finite projective dimension and $N$ be maximal  Cohen-Macaulay. Then  $\CMdef(M\otimes_R N)=\CMdef(M)$.
\end{corollary}

\begin{proof}
Recall from Proposition \ref{re}$(\ast)$    that $\depth_R (M\otimes_R  N)  =\dim(R)-\pd_R (M)$. Since $R$ is 
Cohen-Macaulay, and in view of Auslander-Buchsbaum formula,
$\depth_R (M\otimes_R  N) =\depth(M).$ We combine this with Fact \ref{YD} to observe that
 $\dim (M\otimes_R N) -\depth (M\otimes_R  N)= \dim(M)-\depth(M).$ By definition, $\CMdef(M\otimes_R N)=\CMdef(M)$. 
\end{proof}

\begin{definition}\label{def}
	Suppose $\pd(M)<\infty$. By \it{perfect defect} of $M$, we mean $\pdef(M):=\pd(M)-\grade(M)$.
\end{definition} 
\begin{lemma}\label{l}
	Let $R$ be formally equidimensional, $M$  of finite projective dimension and $N$  maximal 
	 Cohen-Macaulay. Then $\depth(M\otimes_R N)=\dim(M)-\pdef(M)$.
\end{lemma}

\begin{proof}
The assumptions guarantee that $\dim(R)=\dim(M)+\grade(M)$ (see \cite{ro}) and $\Tor^R_+(M,N)=0$. 
Recall from Fact \ref{aus} that $\depth_R (M\otimes_R  N)   =\depth_R (N)-\pd_R (M)$. Combining these,   we get the desired claim.
\end{proof}

The following is a generalization of Fact \ref{22}:

\begin{corollary}\label{cmc}
A local ring	$R$  is Cohen-Macaulay  iff  $\pdef(M)=\CMdef(M)$  for all $\pd(M)<\infty$. 
\end{corollary}

\begin{proof}
	Suppose $R$  is Cohen-Macaulay and $\pd(M)<\infty$. Apply $N:=R$ in Lemma \ref{l} to see
	$\depth(M)=\dim(M)-\pdef(M)$. By definition, 	$\depth(M)=\dim(M)-\CMdef(M)$.
	Combining these, $\pdef(M)=\CMdef(M)$. Conversely, assume $\pdef(M)=\CMdef(M)$  for all $\pd(M)<\infty$. 
	Apply this for $M:=R$, we see $0=\pdef(R)=\CMdef(R)$, i.e., $R$  is Cohen-Macaulay.
\end{proof}

\section{Being integral domain}

In Corollary \ref{pri} we were interested on primary modules of finite projective dimension.
Symbolic  powers  of primes is the natural source to produce primary modules.
Recall that
the $n$-th symbolic power of $I$ is denoted by $I^{(n)}$.
For example, $\frac{R}{\fp^{(n)}}$ is primary. If it has finite projective dimension
it imposes some conditions on the ring.
Concerning the following observation,  the case $n=1$ is known, see   \cite[Corollary II.3.3]{PS2}.
For more partial results with $n=1$, see \cite[Lemma 4.9]{h} and \cite[Abstract]{d}.

\begin{observation}\label{o} Let $R$ be a  local  ring  and $\fp \in\Spec(R)$.
	If $\pd(R/ \fp^{(n)}) < \infty$ for some  $n\in\mathbb{N}$ and $\fp^{(n)}\neq 0$,    then  $R$ is an integral domain.
\end{observation}

\begin{proof}
	We look at  the localization map $f:R\to R_{\fp}$. Let $r\in\ker(f)$. Then
	there is $x\in R\setminus \fp$ such that $rx=0$. Recall that $\fp^{(n)}$ is $\fp$-primary. Then
	$\frac{R}{\fp^{(n)}}\stackrel{x}\lo \frac{R}{\fp^{(n)}}$ is injective as $x\notin\rad(\fp^{(n)})=\fp$.
	In other words, $x$ is $\frac{R}{\fp^{(n)}}$-sequence.
	Now, we apply the assumption $\pd(\frac{R}{\fp^{(n)}}) < \infty$ along with Auslander's zero-divisor conjecture
	(follows from Roberts' intersection theorem) to observe that $x$ is $R$-sequence. From this $r=0$, i.e., 
	$f$ is injective. Now recall that
	$\fp^{(n)}=\fp^n R_{\fp}
	\cap R$. Thus, $$\fp^{(n)}R_{\fp}=( \fp^n R_{\fp}
	\cap R)R_{\fp}=\fp^n R_{\fp}.$$  Since $\fp^{(n)}\neq 0$ we see $\fp^n R_{\fp}\neq 0.$
	In particular, some nonzero powers of the maximal ideal of $R_{\fp}$ is of finite projective dimension. By a result of Levin and Vasconcelos \cite[Theorem 1.1]{lv} we know  $R_{\fp}$ is regular and local. 
	Regular local rings are domain. From this $R$ is an integral domain.
\end{proof}
The  nonvanishing of $\fp^{(n)}$ is important. For example, let $R:=k[X]/(X^2)$ and $\fp:=(x)$. Then $\fp^2=0$ and consequently, $\fp^{(2)}=0$. So, $\fp^{(2)}$
is free and $R$ is not an integral domain. In this regard, we have:

\begin{corollary}Let $R$ be a  reduced local  ring  and $\fp \in\Spec(R)$.
	If $\pd(R/ \fp^{(n)}) < \infty$ for some  $n\in\mathbb{N}$,    then  $R$ is an integral domain.
\end{corollary} 

\begin{proof} 
	We may assume that $\fp\neq 0$. Since $0\neq\fp^{n}\subset \fp^{(n)}$, we deduce that
	$\fp^{(n)}$ is nonzero and of finite projective dimension, in particular it has a regular element. Since $\rad(\fp^{(n)})=\fp$ it follows that $\fp$
	has a regular element. From this, $\dim(R_{\fp})\geq \depth(R_{\fp})\geq 1$.
	In other words,  $\fp^n R_{\fp}\neq 0$.	Put this in the previous result to see that
	$R$ is an integral domain.
\end{proof}

\begin{remark}
Observation \ref{o} is not true if we replace $\pd(-)$ with $ \Gdim(-)$, e.g. we look at $R:=k[X]/(X^2)$,  $\fp:=(x)$ and $n:=1$. \end{remark}
By $\id_R(-)$ we mean the injective dimension. Here, we present a dual version of Observation \ref{o}:

\begin{corollary}\label{id}
	Let $R$ be a  local  ring  and $\fp \in\Spec(R)$.
	If $\id_R(R/\fp^{(n)}) < \infty$ for some  $n\in\mathbb{N}$  and $\fp^{(n)}\neq 0$,   then  $R$ is a (Gorenstein) integral domain.
\end{corollary} 

\begin{proof} 
	By an application of Bass' conjecture (follows by Rorbert's  intersection theorem), $R$ is Gorenstein. 
	For more details, see \cite[Corollary II.5.3]{PS2}. Over Gorenstein rings,
	a module has finite injective dimension if and only if it has finite projective dimension.
	From this, $\pd(R/\fp^{(n)}) < \infty$. In view of Observation \ref{o}, $R$ is an integral domain.
\end{proof}

\begin{corollary}\label{id}
	Let $R$ be a reduced local  ring  and $\fp \in\Spec(R)$.
	If $\id_R(R/\fp^{(n)}) < \infty$ for some  $n\in\mathbb{N}$,   then  $R$ is a (Gorenstein) integral domain.
\end{corollary}

\begin{example} There is a ring with  a prime ideal $\fp$ such that the following assertions holds:
\begin{enumerate}
	\item[i)]   It may be $\id(\fp)<\infty$ and $\id(\fp^{(2)})=\infty$.
	\item[ii)] It may be $\pd(\fp)=\infty$ and $\pd(\fp^{(2)})<\infty$.
\end{enumerate}
\end{example}

\begin{proof} 
	Let $X:=(X_{ij})$ be a symmetric $n \times n$ matrix of indeterminates and let
	$I_{r+1}(X)$ be the ideal generated  by the $(r + 1)$-minors of $X$. The rings $R = k[X]/I_{r+1}(X)$ is Cohen-Macaulay. Let $\fp$ be the ideal generated by the $r$-minors
	of the first $r$ rows of $X$. It is well-known that $\fp^{(2)}$ is principal.
Suppose $r+1\equiv_2 n$ is not true. Then $\fp$ is the canonical module of $R$, see \cite[Exercise 7.3.10]{BH}.

i) Since   $\fp$ is the canonical module we have  $\id(\fp)<\infty$. Since $\fp^{(2)}$ is principal (\cite[Exercise 7.3.10]{BH}) and $S_{r+1}$ is an integral domain,
$\id(\fp^{(2)})=\id(R)=\infty$, because $R$ is not Gorenstein.

ii) If $\pd(\omega_R)<\infty$ it follows from Auslander-Bridger formula that
$\omega_R$ is free, and so $R$ is Gorenstein. This contraction shows that $\pd(\omega_R)=\infty$.
Since $\fp^{(2)}$ is principal   and $R$ is an integral domain,
$\pd(\fp^{(2)})=\pd(R)=0<\infty$
\end{proof}

\begin{example} For each  $t  >1$ there is a ring with  a prime ideal $\fp$ such that  $\id(\fp)=\infty$ and $\id(\fp^{(t)})=\id(\fp^{t})<\infty$.
\end{example}

\begin{proof} For simplicity we assume $t=2$.
	Let $X:=(X_{ij})$ be an  $n \times (n+2)$ matrix of indeterminates and let
	$I_{r+1}(X)$ be the ideal generated  by the $(r + 1)$-minors of $X$. The  ring $R: = k[X]/I_{r+1}(X)$ is Cohen-Macaulay. Let $\fp$ be the ideal generated by the $r$-minors
	of the first $r$ rows of $X$. It is well-known that $\fp^{(2)}$  is the canonical module of $R$, see \cite[Theorem 7.3.6]{BH}.

  Since   $\fp^{(2)}$ is the canonical module we have  $\id(\fp^{(2)})<\infty$.
	Suppose on the way of contradiction that 	$\id(\fp )<\infty$. 
	 Since $\fp$ is a height one prime ideal
	 and $R/\fp $ is a Cohen-Macaulay ring, it follows from
	 $$0\lo \fp\lo R\lo R/\fp\lo 0$$ that $\fp$ is maximal  Cohen-Macaulay.
	 We apply these to find an integer $\ell$ such that $\fp\cong \oplus_{\ell}\omega_R=\oplus_{\ell}\fp^{(2)}$.
	 Taking the rank, it follows that $\ell=1$. Consequently, $\fp\cong \fp^{(2)}$. Now, we look at
	 them in the divisor class  group. We know that $[\fp^{(2)}]=2[\fp]$. In particular, $\fp$ and $\fp^{(2)}$ are not isomorphic. This contradiction says that $\id(\fp)=\infty$.
\end{proof}


\end{document}